\documentclass[oneside,canadian,english]{amsart}
\usepackage[T1]{fontenc}
\usepackage[latin9]{inputenc}
\usepackage{geometry}
\usepackage{mathrsfs}
\usepackage{amsthm}
\usepackage{amssymb}
\usepackage{esint}

\makeatletter
\numberwithin{equation}{section}
\numberwithin{figure}{section}
\theoremstyle{plain}
\newtheorem{thm}{\protect\theoremname}
  \theoremstyle{remark}
  \newtheorem*{rem*}{\protect\remarkname}
  \theoremstyle{plain}
  \newtheorem{lem}[thm]{\protect\lemmaname}

\makeatother

\usepackage{babel}
  \addto\captionscanadian{\renewcommand{\lemmaname}{Lemma}}
  \addto\captionscanadian{\renewcommand{\remarkname}{Remark}}
  \addto\captionscanadian{\renewcommand{\theoremname}{Theorem}}
  \addto\captionsenglish{\renewcommand{\lemmaname}{Lemma}}
  \addto\captionsenglish{\renewcommand{\remarkname}{Remark}}
  \addto\captionsenglish{\renewcommand{\theoremname}{Theorem}}
  \providecommand{\lemmaname}{Lemma}
  \providecommand{\remarkname}{Remark}
\providecommand{\theoremname}{Theorem}

\begin{document}

\title{Convergence bound in total variation for an image restoration model}

\author{Oliver Jovanovski\\
}

\address{Department of Mathematics and Statistics, 4700 Keele Street, York
University, M3J 1P3, Canada}

\email{olijovan@mathstat.yorku.ca\\
\emph{Telephone:} (416) 736-2100 Extension 40616 }
\begin{abstract}
We consider a stochastic image restoration model proposed by A. Gibbs
(2004), and give an upper bound on the time it takes for a Markov
chain defined by this model to be $\epsilon$- close in total variation
to equilibrium. We use Gibbs' result for convergence in the Wasserstein
metric to arrive at our result. Our bound for the time to equilibrium
of similar order to that of Gibbs.
\end{abstract}


\maketitle

\section{Introduction}

A.L. Gibbs \cite{key-1} introduced a stochastic image restoration
model for an $N$ pixel \foreignlanguage{canadian}{greyscale} image
$x=\left\{ x_{i}\right\} _{i=1}^{N}$. More specifically, in this
model each pixel $x_{i}$ corresponds to a real value in $\left[0,1\right]$,
where a black pixel is represented by $0$ and a white pixel is represented
by the value $1$. It is assumed that in the real-world space of such
images, each pixel tends to be like its nearest \foreignlanguage{canadian}{neighbours}
(in the absence of any evidence otherwise). This assumption is expressed
in the prior probability density of the image , which is given by
\begin{equation}
\pi_{\gamma}\left(x\right)\propto exp\left\{ -\sum_{\left\langle i,j\right\rangle }\frac{1}{2}\left[\gamma\left(x_{i}-x_{j}\right)\right]^{2}\right\} \label{eq:prior}
\end{equation}
on the state space $\left[0,1\right]^{N}$, and is equal to $0$ elsewhere.
The sum in (\ref{eq:prior}) is over all pairs of pixels that are
considered to be \foreignlanguage{canadian}{neighbours}, and the parameter
$\gamma$ represents the strength of the assumption that \foreignlanguage{canadian}{neighbouring}
pixels are similar. Here images are assumed to have an underlying
graph structure. The familiar 2-dimensional digital image is a special
case, where \foreignlanguage{canadian}{usually} one might assume that
the \foreignlanguage{canadian}{neighbours} of a pixel $x_{i}$ in
the interior of the image (i.e. $x_{i}$ not on the boundary of the
image) are the 4 or 8 pixels surrounding $x_{i}$, depending on whether
or not we decide to consider the 4 pixels diagonal to $x_{i}$. 

The actual observed image $y=\left\{ y_{i}\right\} _{i=1}^{N}$ is
assumed to be the result of the original image subject to distortion
by random noise, with every pixel modified independently through the
addition of a $Normal\left(0,\sigma^{2}\right)$ random variable (hence
$y_{i}\in\mathbb{R}$). The resulting posterior probability density
for the original image is given by 
\begin{equation}
\pi_{posterior}\left(x\left|y\right.\right)\propto exp\left\{ -\sum_{i=1}^{N}\frac{1}{2\sigma^{2}}\left(x_{i}-y_{i}\right)^{2}-\sum_{\left\langle i,j\right\rangle }\frac{1}{2}\left[\gamma\left(x_{i}-x_{j}\right)\right]^{2}\right\} \label{eq:posterior}
\end{equation}
supported on $\left[0,1\right]$.

Samples from (\ref{eq:posterior}) can be approximately obtained by
means of a Gibbs sampler. In this instance, the algorithm works as
follows: at every iteration the sampler \foreignlanguage{canadian}{chooses}
a site $i$ uniformly at random, and replaces the value $x_{i}$ at
this location according to the full conditional density at that site.
This density is given by 
\begin{eqnarray}
\pi_{FC}\left(x_{i}\left|y,x_{k\neq i}\right.\right) & \propto & exp\left\{ \frac{\left(\sigma^{-2}+n_{i}\gamma^{2}\right)}{2}\right.\label{eq:pifc}\\
 &  & \left.\cdot\left[x_{i}-\left(\sigma^{-2}+n_{i}\gamma^{2}\right)^{-1}\left(\sigma^{-2}y_{i}+\gamma^{2}\sum_{j\sim i}x_{j}\right)\right]^{2}\right\} \nonumber 
\end{eqnarray}
on $\left[0,1\right]$ and $0$ elsewhere. Here $n_{i}$ is the number
of \foreignlanguage{canadian}{neighbours} the $i^{th}$ pixel has,
and $j\sim i$ indicates that the $j^{th}$ pixel is one of them.
It follows that (\ref{eq:pifc}) is a restriction of a\\
$Normal\left(\left(\sigma^{-2}+n_{i}\gamma^{2}\right)^{-1}\left(\sigma^{-2}y_{i}+\gamma^{2}\sum_{j\sim i}x_{j}\right),\left(\sigma^{-2}+n_{i}\gamma^{2}\right)^{-1}\right)$
distribution to the set $\left[0,1\right]$. 

The bound on the rate of convergence to equilibrium given in \cite{key-1}
is stated in terms of the Wasserstein metric $d_{W}$. This is defined
as follows: if $\mu_{1}$ and $\mu_{2}$ are two probability measures
on the same state space which is endowed with some metric $d$, then
\[
d_{W}\left(\mu_{1},\mu_{2}\right):=inf\mathbb{E}\left[d\left(\xi_{1},\xi_{2}\right)\right]
\]
where the infimum is taken over all joint distributions $\left(\xi_{1},\xi_{2}\right)$
such that $\xi_{1}\sim\mu_{1}$ and $\xi_{2}\sim\mu_{2}$. 

Another commonly used metric for measuring the distance of a Markov
chain from its equilibrium distribution is the total variation metric,
defined for two probability measures $\mu_{1}$ and $\mu_{2}$ on
the state space $\Omega$ by 
\[
d_{TV}\left(\mu_{1},\mu_{2}\right):=sup\left|\mu_{1}\left(A\right)-\mu_{2}\left(A\right)\right|
\]
where the supremum is taken over all measurable $A\subseteq\Omega$. 

The underlying metric on the state space used throughout \cite{key-1}
(and hence used implicitly in the statement of Theorem \ref{thm:0})
is defined by $d\left(x,y\right):=\sum_{i}n_{i}\left|x_{i}-z_{i}\right|$.
This is a non-standard choice for a metric on $\left[0,1\right]^{N}$,
however it is comparable to the more usual $l_{1}$ taxicab metric
$\hat{d}\left(x,y\right):=\sum_{i}\left|x_{i}-z_{i}\right|$ since
\[
n_{min}\cdot\hat{d}\left(x,y\right)\leq d\left(x,y\right)\leq n_{max}\cdot\hat{d}\left(x,y\right)
\]
where $n_{max}:=max_{i}\left\{ n_{i}\right\} $ and $n_{min}:=min_{i}\left\{ n_{i}\right\} $.
Hence, for two probability measures $\mu_{1}$ and $\mu_{2}$ on $\left[0,1\right]^{N}$,
it follows immediately that 
\[
n_{min}\cdot d_{\hat{W}}\left(\mu_{1},\mu_{2}\right)\leq d_{W}\left(\mu_{1},\mu_{2}\right)\leq n_{max}\cdot d_{\hat{W}}\left(\mu_{1},\mu_{2}\right)
\]
where $d_{\hat{W}}$ and $d_{W}$ are the Wasserstein metrics associated
with $\hat{d}$ and $d$ respectively. 

If $\Theta_{1}$ and $\Theta_{2}$ are two random variables on the
same state space with probability measures $m_{1}$ and $m_{2}$ respectively,
then we shall write 
\[
d_{W}\left(\Theta_{1},\Theta_{2}\right):=d_{W}\left(m_{1},m_{2}\right)\quad\mathrm{and}\quad d_{TV}\left(\Theta_{1},\Theta_{2}\right):=d_{TV}\left(m_{1},m_{2}\right)
\]

\medskip{}

Gibbs \cite{key-1} shows that 
\begin{thm}
\cite{key-1}\label{thm:0} Let $X^{t}$ be a copy of the Markov chain
evolving according to the Gibbs sampler, and let $Z^{t}$ be a chain
in equilibrium, distributed according to $\pi_{posterior}$. Then
if $\left[0,1\right]^{N}$ is given the metric $d\left(x,y\right):=\sum_{i}n_{i}\left|x_{i}-z_{i}\right|$,
it follows that $d_{W}\left(X^{t},Z^{t}\right)\leq\epsilon$ whenever
\begin{equation}
t>\vartheta\left(\epsilon\right):=\frac{log\left(\frac{\epsilon}{n_{max}N}\right)}{log\left(1-N^{-1}\left(1+n_{max}\gamma^{2}\sigma^{2}\right)^{-1}\right)}\label{eq:Wmixtime}
\end{equation}

\end{thm}
By the comments preceding the statement of this theorem, (\ref{eq:Wmixtime})
remains true with the standard $l_{1}$ metric on the state space,
if we replace $\epsilon$ by $n_{min}\cdot\epsilon$ in the right-hand
side of this inequality.
\begin{rem*}
Equation (\ref{eq:Wmixtime}) appears in \cite{key-1} with the denominator
being\\
$log\left(N-1/N+n_{max}N^{-1}\gamma^{2}\left(\sigma^{-2}+n_{max}\gamma^{2}\right)^{-1}\right)$.
It is obvious from their proof that this is a typographical error,
and that the term $N-1/N$ was intended to be $\left(N-1\right)/N$.
\end{rem*}
It is not difficult to see that $d_{TV}$ is a special case of $d_{W}$
when the underlying metric is given by $d\left(x,z\right)=1$ if $x\neq z$.
In general however, convergence in $d_{W}$ does not imply convergence
in $d_{TV}$, and vice versa (see \cite{key-2} for examples where
convergence fails, as well as some conditions under which convergence
in one of $d_{W}$, $d_{TV}$ implies convergence in the other). The
purpose of this paper is to obtain a bound in $d_{TV}$ by making
use of (\ref{eq:Wmixtime}) and simple properties of the Markov chain,
without specifically engaging in a new study of the mixing time. 

Let $X_{t}$ be a copy of the Markov chain, and let $\mu^{t}$ be
its probability distribution. Furthermore, define $\zeta_{i}:=\left(\sigma^{-2}+n_{i}\gamma^{2}\right)^{-1}\left(\sigma^{-2}y_{i}+\gamma^{2}n_{max}\right)$,
$\zeta:=max\left\{ \left|\zeta_{i}\right|\right\} $ and $\tilde{\sigma_{i}}^{2}=\left(\sigma^{-2}+n_{i}\gamma^{2}\right)^{-1}$.
If $\pi$ is the posterior distribution with density function $\pi_{posterior}$,
we show that
\begin{thm}
\label{thm:1}Let $X_{t}$ be a copy of the Markov chain evolving
according to the Gibbs sampler, and let $Z^{t}$ be a chain in equilibrium.
Then $d_{TV}\left(X^{t},Z^{t}\right)\leq\epsilon$ whenever
\begin{equation}
t>\vartheta\left(\omega^{2}\right)+M\label{eq:tvmixtime}
\end{equation}
where $M=\left\lceil Nlog\left(N\right)+Nlog\left(\frac{2}{\epsilon}\right)\right\rceil $
and $\omega=\left[1-\left(1-\frac{\epsilon}{2}\right)^{M^{-1}}\right]/\left(1+e^{\frac{\left(\zeta+1\right)^{2}}{2\tilde{\sigma}^{2}}}\right)$
.
\end{thm}
Akin to the bound for the metric $d_{W}$, this bound is also $O\left(Nlog\frac{N}{\epsilon}\right)$.
A notable difference, however, is that in our bound there is a (quadratic)
dependence on $\zeta$ (and hence a quadratic dependence on $max\left\{ \left|y_{i}\right|\right\} $). 

Since this state space is bounded, it also easily follows (using previously
defined notation) that $d_{\hat{W}}\left(\mu_{1},\mu_{2}\right)\leq N\cdot d_{TV}\left(\mu_{1},\mu_{2}\right)$
and $d_{W}\left(\mu_{1},\mu_{2}\right)\leq n_{max}\cdot N\cdot d_{TV}\left(\mu_{1},\mu_{2}\right)$.
Therefore, Theorem \ref{thm:1} also implies a bound in $d_{W}$ as
well as $d_{\hat{W}}$. 

Section \ref{sec:2} will present the proof of Theorem \ref{thm:1},
and will conclude with a discussion of the proof strategy.

\medskip{}

\section{\label{sec:2}From $d_{W}$ to $d_{TV}$}

Let $t$ be some fixed time, and let $X^{s}$ and $\mathit{Z}^{s}$
($s=1,\ldots,t$) be two instances of the Markov chain, evolving as
defined in the lines preceding (\ref{eq:pifc}). The coupling method
\cite{key-3} allows us to bound total variation via the inequality
\[
d_{TV}\left(X^{t},Z^{t}\right)\leq\mathbb{P}\left[X^{t}\neq Z^{t}\right].
\]
Having uniformly selected $i$ from $\left\{ 1,\ldots,N\right\} $,
we couple the pixel $X_{i}^{t+1}$ with $Z_{i}^{t+1}$ as follows:
let $f_{i}$ and $g_{i}$ be the conditional density functions of
$X_{i}^{t+1}$ given $X^{t}$ and of $Z_{i}^{t+1}$given $Z^{t}$,
respectively. Choose a point $\left(a_{1},a_{2}\right)$ uniformly
from the area defined by $A_{X}=\left\{ \left(a,b\right)|f_{i}\left(a\right)>0,0\leq b\leq f_{i}\left(a\right)\right\} $
- i.e. the area under the graph of $f_{i}$, and set $X_{i}^{t+1}=a_{1}$.
If the point $\left(a_{1},a_{2}\right)$ is also in the set $A_{Z}=\left\{ \left(a,b\right)|g_{i}\left(a\right)>0,0\leq b\leq g_{i}\left(a\right)\right\} $,
then set $Z_{i}^{t+1}=X_{i}^{t+1}=a_{1}$. Otherwise $\left(a_{1},a_{2}\right)\in A_{x}\backslash A_{z}$,
and in this case choose a point $\left(b_{1},b_{2}\right)$ uniformly
from $A_{Z}\backslash A_{X}=\left\{ \left(a,b\right)\left|g_{i}\left(a\right)\geq b\geq f_{i}\left(a\right)\right.\right\} $
and set $Z_{i}^{t+1}=b_{1}$. Observe that $X^{s}$ and $Z^{s}$ ($s=0,\ldots,t+1$)
are indeed two faithful copies of the Markov chain.

In order to proceed, we will establish the following results. 
\begin{lem}
\label{lem:2}Let $U_{1}\sim Normal\left(\mu_{1},\sigma^{2}\right)$
and $U_{2}\sim Normal\left(\mu_{2},\sigma^{2}\right)$, and let $W_{1}$
and $W_{2}$ have the distributions of $U_{1}$ and $U_{2}$ conditioned
to be in some measurable set $\mathit{S}$. Let $f_{U_{1}}$, $f_{U_{2}}$,
$f_{W_{1}}$ and $f_{W_{2}}$ be their respective density functions.
Then
\[
d_{TV}\left(W_{1},W_{2}\right)\leq\frac{d_{TV}\left(U_{1},U_{2}\right)}{min\left(\intop_{S}f_{U_{1}},\intop_{S}f_{U_{2}}\right)}
\]
\end{lem}
\begin{proof}
We start by noting that 

\begin{eqnarray}
d_{TV}\left(W_{1},W_{2}\right) & = & \int_{f_{W_{1}}\geq f_{W_{2}}}\left(f_{W_{1}}-f_{W_{2}}\right)\label{eq:lem1ineq1}\\
 & = & \int_{f_{W_{1}}\geq f_{W_{2}}}\left(\frac{f_{U_{1}}}{\intop_{S}f_{U_{1}}}-\frac{f_{U_{2}}}{\intop_{S}f_{U_{2}}}\right)\nonumber 
\end{eqnarray}
The first equality is one of a few different equivalent definitions
of total variation. A proof is given in Proposition 3 of \cite{key-4}. 

Now if $\intop_{S}f_{U_{1}}\geq\intop_{S}f_{U_{2}}$, then the above
is bounded by
\begin{eqnarray}
d_{TV}\left(W_{1},W_{2}\right) & \leq & \frac{1}{\intop_{S}f_{U_{2}}}\int_{f_{W_{1}}\geq f_{W_{2}}}\left(f_{U_{1}-}f_{U_{2}}\right)\label{eq:lem1ineq2}\\
 & \leq & \frac{1}{\intop_{S}f_{U_{2}}}\int_{f_{U_{1}}\geq f_{U_{2}}}\left(f_{U_{1}-}f_{U_{2}}\right)\nonumber \\
 & = & \frac{d_{TV}\left(U_{1},U_{2}\right)}{min\left(\intop_{S}f_{U_{1}},\intop_{S}f_{U_{2}}\right)}\nonumber 
\end{eqnarray}
The second inequality follows from the observation that
\[
\frac{f_{U_{1}}\left(w\right)}{\intop_{S}f_{U_{1}}}\geq\frac{f_{U_{2}}\left(w\right)}{\intop_{S}f_{U_{2}}}\Rightarrow\frac{f_{U_{1}}\left(w\right)}{\intop_{S}f_{U_{2}}}\geq\frac{f_{U_{2}}\left(w\right)}{\intop_{S}f_{U_{2}}}\Rightarrow f_{U_{1}}\left(w\right)\geq f_{U_{2}}\left(w\right)
\]
Similarly, if $\intop_{S}f_{U_{2}}\geq\intop_{S}f_{U_{1}}$, then
we repeat the same argument with 
\[
d_{TV}\left(W_{1},W_{2}\right)=\int_{f_{W_{2}}\geq f_{W_{1}}}\left(f_{W_{2}}-f_{W_{1}}\right)
\]
in place of (\ref{eq:lem1ineq1}), arriving at the same result.
\end{proof}
A simple but useful result is the following lemma: 
\begin{lem}
\label{lem:3}$\left(2\pi\sigma^{2}\right)^{-1/2}\int_{0}^{1}e^{\frac{-\left(x-\zeta_{i}\right){}^{2}}{2\sigma^{2}}}\geq\left(2\pi\sigma^{2}\right)^{-1/2}e^{-\frac{\left(\left|\zeta_{i}\right|+1\right)^{2}}{2\sigma^{2}}}$\end{lem}
\begin{proof}
This is trivial, since $\left(\left|\zeta_{i}\right|+1\right)\geq\left|x-\zeta_{i}\right|$
for any $x\in\left[0,1\right]$.
\end{proof}
Now let $U_{1}\sim Normal\left(\left(\sigma^{-2}+n_{i}\gamma^{2}\right)^{-1}\left(\sigma^{-2}y_{i}+\gamma^{2}\sum_{j\sim i}x_{j}^{t}\right),\tilde{\sigma_{i}}^{2}\right)$
and\\
$U_{2}\sim Normal\left(\left(\sigma^{-2}+n_{i}\gamma^{2}\right)^{-1}\left(\sigma^{-2}y_{i}+\gamma^{2}\sum_{j\sim i}z_{j}^{t}\right),\tilde{\sigma_{i}}^{2}\right)$.
Applying Lemma \ref{lem:2} to $\left(X_{i}^{t+1},Z_{i}^{t+1}\right)$
with $S=\left[0,1\right]$, we see that conditional on $\mathscr{F}_{t}$
(sigma algebra generated by $X^{t}$ and $Z^{t}$) 
\begin{eqnarray}
\mathbb{P}\left[X_{i}^{t+1}\neq Z_{i}^{t+1}\left|\mathscr{F}_{t}\right.\right] & = & d_{TV}\left(X_{i}^{t+1},Z_{i}^{t+1}\left|\mathscr{F}_{t}\right.\right)\nonumber \\
 & \leq & \frac{d_{TV}\left(U_{1},U_{2}\left|\mathscr{F}_{t}\right.\right)}{min\left(\intop_{S}f_{U_{1}},\intop_{S}f_{U_{2}}\right)}\nonumber \\
 & \leq & \left(2\pi\tilde{\sigma_{i}}^{2}\right)^{1/2}e^{\frac{\left(\left|\zeta_{i}\right|+1\right)^{2}}{2\tilde{\sigma_{i}}^{2}}}d_{TV}\left(U_{1},U_{2}\left|\mathscr{F}_{t}\right.\right)\label{eq:dtvxz}
\end{eqnarray}
For the second inequality we have used Lemma \ref{lem:3}. By Lemma
15 of \cite{key-2} it follows that 
\begin{equation}
d_{TV}\left(U_{1},U_{2}\left|\mathscr{F}_{t}\right.\right)\leq\frac{\left|\mathbb{E}\left[U_{1}\left|\mathscr{F}_{t}\right.\right]-\mathbb{E}\left[U_{2}\left|\mathscr{F}_{t}\right.\right]\right|}{\sqrt{2\pi\tilde{\sigma_{i}}^{2}}}\label{eq:u1u2}
\end{equation}

\medskip{}
Hence by (\ref{eq:dtvxz})
\begin{eqnarray}
\mathbb{P}\left[X_{i}^{t+1}\neq Z_{i}^{t+1}\left|\mathscr{F}_{t}\right.\right] & \leq & e^{\frac{\left(\left|\zeta_{i}\right|+1\right)^{2}}{2\tilde{\sigma_{i}}^{2}}}\left|\mathbb{E}\left[U_{1}\left|\mathscr{F}_{t}\right.\right]-\mathbb{E}\left[U_{2}\left|\mathscr{F}_{t}\right.\right]\right|\nonumber \\
 & = & e^{\frac{\left(\left|\zeta_{i}\right|+1\right)^{2}}{2\tilde{\sigma_{i}}^{2}}}\tilde{\sigma_{i}}^{2}\gamma^{2}\left|\sum_{j\sim i}X_{j}^{t}-\sum_{j\sim i}Z_{j}^{t}\right|\nonumber \\
 & \leq & e^{\frac{\left(\left|\zeta_{i}\right|+1\right)^{2}}{2\tilde{\sigma_{i}}^{2}}}\tilde{\sigma_{i}}^{2}\gamma^{2}\sum_{j\sim i}\left|X_{j}^{t}-Z_{j}^{t}\right|\label{eq:dtvxz2}
\end{eqnarray}
We can now proceed with the proof of Theorem \ref{thm:1}.
\begin{proof}[Proof of Theorem \ref{thm:1}]
 Let $\epsilon>0$ be given, and define $\tilde{\epsilon}:=1-\left(1-\frac{\epsilon}{2}\right)^{M^{-1}}$
(recall that $M=\left\lceil Nlog\left(N\right)+Nlog\left(\frac{2}{\epsilon}\right)\right\rceil $)
and $\omega:=\tilde{\epsilon}/\left(1+e^{\frac{\left(\zeta+1\right)^{2}}{2\tilde{\sigma}^{2}}}\right)$
with $\tilde{\sigma}:=min\left\{ \tilde{\sigma}_{i}\right\} $. By
Theorem \ref{thm:0}, $d_{W}\left(X^{t},Z^{t}\right)\leq\omega^{2}$
whenever $t\geq\tau:=\left\lceil log\left(\frac{\omega^{2}}{n_{max}N}\right)/log\left(1-N^{-1}\left(1+\sigma^{2}n_{max}\gamma^{2}\right)^{-1}\right)\right\rceil $.
Since the infimum in the definition of $d_{W}$ is achieved (see for
example Section 5.1 of \cite{key-6}), we can find a joint distribution
$\mathcal{L}\left(u^{\tau},v^{\tau}\right)$ of two random variables
$u^{\tau}\sim X^{\tau}$ and $v^{\tau}\sim Z^{\tau}$, such that $\mathbb{E}\left[d\left(u^{\tau},v^{\tau}\right)\right]=\mathbb{E}\left[\sum n_{i}\left|u_{i}^{\tau}-v_{i}^{\tau}\right|\right]\leq\omega^{2}$
(we use the superscript $\tau$ in $u^{\tau}$ and $v^{\tau}$ to
preserve notational consistency with $X^{\tau}$ and $Z^{\tau}$).
And by Markov's inequality we get
\begin{eqnarray}
\mathbb{P}\left[\sum_{k\sim j}\left|u_{k}^{\tau}-v_{k}^{\tau}\right|\geq\omega\: for\: some\: j\right] & \leq & \mathbb{P}\left[d\left(u^{\tau},v^{\tau}\right)\geq\omega\right]\nonumber \\
 & \leq & \omega\label{eq:xkzkomega}
\end{eqnarray}

\noindent For $s=1,\ldots$ , define the Markov chains $u^{\tau+s}\sim X^{\tau+s}$
and $v^{\tau+s}\sim Z^{\tau+s}$ by uniformly choosing (for every
$s$) a site $i$ and assigning values to $\left(u_{i}^{\tau+s},v_{i}^{\tau+s}\right)$
as described at the beginning of Section \ref{sec:2}. Note that $d_{TV}\left(u^{\tau+s},v^{\tau+s}\right)=d_{TV}\left(X^{\tau+s},Z^{\tau+s}\right)$,
hence it suffices to show that $d_{TV}\left(u^{\tau+s},v^{\tau+s}\right)\leq\epsilon$
whenever $\vartheta\left(\omega^{2}\right)+M$. By splitting up the
above probability and applying (\ref{eq:dtvxz2}) and (\ref{eq:xkzkomega}),
we conclude that at the chosen site $i$
\begin{eqnarray}
\mathbb{P}\left[u_{i}^{\tau+1}\neq v_{i}^{\tau+1}\right] & = & \mathbb{P}\left[u_{i}^{\tau+1}\neq v_{i}^{\tau+1}\left|\sum_{k\sim i}\left|u_{k}^{\tau}-v_{k}^{\tau}\right|<\omega\right.\right]\cdot\mathbb{P}\left[\sum_{k\sim i}\left|u_{k}^{\tau}-v_{k}^{\tau}\right|<\omega\right]\nonumber \\
 &  & +\mathbb{P}\left[u_{i}^{\tau+1}\neq v_{i}^{\tau+1}\left|\sum_{k\sim i}\left|u_{k}^{\tau}-v_{k}^{\tau}\right|\geq\omega\right.\right]\cdot\mathbb{P}\left[\sum_{k\sim i}\left|u_{k}^{\tau}-v_{k}^{\tau}\right|\geq\omega\right]\nonumber \\
 & \leq & e^{\frac{\left(\left|\zeta_{i}\right|+1\right)^{2}}{2\tilde{\sigma_{i}}^{2}}}\tilde{\sigma_{i}}^{2}\gamma^{2}\omega+\omega\nonumber \\
 & \leq & \omega\left(e^{\frac{\left(\zeta+1\right)^{2}}{2\tilde{\sigma}^{2}}}+1\right)\nonumber \\
 & = & \tilde{\epsilon}\label{eq:epstilde}
\end{eqnarray}
Let $i_{m}$ be the pixel chosen at time $\tau+m$ for $m=1,2,\ldots$
. For $j\geq1$, define the events \textbf{$B_{j}:=\left\{ u_{i_{j}}^{\tau+j}=v_{i_{j}}^{\tau+j}\right\} $}
and \textbf{$B_{0}:=\left\{ d\left(u^{\tau},v^{\tau}\right)\leq\omega\right\} $},
and observe that in the event $\left\{ \bigcap_{k=0}^{j}B_{k}\right\} $,
we have $d\left(u^{\tau+j},v^{\tau+j}\right)\leq d\left(u^{\tau},v^{\tau}\right)\leq\omega$.
Therefore by equations (\ref{eq:dtvxz2}) and (\ref{eq:xkzkomega})
\begin{eqnarray*}
\mathbb{P}\left[u_{i_{m}}^{\tau+m}\neq v_{i_{m}}^{\tau+m}\left|\bigcap_{k=1}^{m-1}B_{k}\right.\right] & \leq & \mathbb{P}\left[u_{i_{m}}^{\tau+m}\neq v_{i_{m}}^{\tau+m}\left|\bigcap_{k=0}^{m-1}B_{k}\right.\right]\mathbb{P}\left[B_{0}\right]+\omega\\
 & \leq & \omega\left(e^{\frac{\left(\zeta+1\right)^{2}}{2\tilde{\sigma}^{2}}}+1\right)\\
 & = & \tilde{\epsilon}
\end{eqnarray*}
By induction on $m$ we get that 
\begin{eqnarray}
\mathbb{P}\left[\bigcap_{j=1}^{m}B_{j}\right] & \geq & \mathbb{P}\left[B_{m}\left|\bigcap_{j=1}^{m-1}B_{j}\right.\right]\cdot\mathbb{P}\left[\bigcap_{j=1}^{m-1}B_{j}\right]\label{eq:nomiss}\\
 & \geq & \left(1-\tilde{\epsilon}\right)^{m}\nonumber 
\end{eqnarray}
Note that the case $m=1$ follows directly from (\ref{eq:epstilde}).
We will now refer to the 'coupon collector' problem, discussed in
section 2.2 of \cite{key-5}: if $\theta$ is the first time when
a coupon collector has obtained all $N$ out of $N$ coupons, then 

\begin{equation}
\mathbb{P}\left[\theta>M\right]\leq\frac{\epsilon}{2}\label{eq:coupon}
\end{equation}
Let $\phi:=\tau+M$ and let $\theta:=min\left\{ l\geq1:\,\left\{ 1,\ldots,N\right\} \subseteq\left\{ i_{1},\ldots,i_{l}\right\} \right\} $
- i.e. $\tau+\theta$ is the first time when every pixel site has
been chosen at least once after $\tau$. Recall also that $\tilde{\epsilon}:=1-\left(1-\frac{\epsilon}{2}\right)^{M^{-1}}$.
Then
\begin{eqnarray}
\mathbb{P}\left[u^{\phi}\neq v^{\phi}\right] & = & \mathbb{P}\left[u^{\phi}\neq v^{\phi}\left|\theta>M\right.\right]\cdot\mathbb{P}\left[\theta>M\right]+\mathbb{P}\left[u^{\phi}\neq v^{\phi},\theta\leq M\right]\nonumber \\
 & \leq & \mathbb{P}\left[\theta>M\right]+\mathbb{P}\left[u_{i_{j}}^{\tau+j}\neq v_{i_{j}}^{\tau+j}\: for\: some\:1\leq j\leq M\right]\nonumber \\
 & = & \mathbb{P}\left[\theta>M\right]+1-\mathbb{P}\left[\bigcap_{j=1}^{M}B_{j}\right]\label{eq:condnomiss}\\
 & \leq & \frac{\epsilon}{2}+1-\left(1-\tilde{\epsilon}\right)^{M}\nonumber \\
 & = & \frac{\epsilon}{2}+1-\left(\left(1-\frac{\epsilon}{2}\right)^{M^{-1}}\right)^{M}\nonumber \\
 & \leq & \epsilon\nonumber 
\end{eqnarray}
 This proves the statement of the theorem.\end{proof}
\begin{rem*}
The strategy here was to couple two copies of the Markov chain until
favourable conditions were met (i.e. until their Wasserstein distance
was sufficiently small), and then attempt to force coalescence in
``one shot'' at each co-ordinate. This method is described in \cite{key-4}
and \cite{key-7} in a more general context. 

The proof of Theorem \ref{thm:1} is quite specialized, as it involves
the use of specific properties related to this model. We showed that
coalescence between the two chains, one co-ordinate at a time and
without any ``misses'', would occur with high likelihood. One important
property required in order to bound $d_{TV}$ in terms of $d_{W}$,
was bounding the conditional total variation at every co-ordinate
(equivalent to the non-overlapping area under the conditional density
functions at each co-ordinate) in terms of the distance between the
two chains. Another, less stringent, requirement was for the distance
between the two chains not to increase if coalescence was successful
at any co-ordinate (presumably one could construct a metric where
this is not necessarily true). With these conditions satisfied, it
may be possible to apply the ideas of this paper (as well as those
presented in \cite{key-7} and \cite{key-2}) to convert Wasserstein
bounds into TV bounds in a variety of situations. 

\end{rem*}

\subsubsection*{Acknowledgement}

This research was supported in part by the NSERC Discovery Grant of
Neal Madras at York University.

\end{document}